\documentclass{amsart}
\usepackage[utf8]{inputenc}
\usepackage{amsmath}
\usepackage{amssymb}
\usepackage{amsfonts}
\usepackage{cleveref}
\usepackage{tikz-cd}
\usepackage{stmaryrd}

\bibliography{references}

    \usepackage{etoolbox}
    \patchcmd{\section}{\scshape}{\large\bfseries}{}{}
    \makeatletter
    \renewcommand{\@secnumfont}{\bfseries}
    \makeatother

\usepackage{MnSymbol}
\DeclareMathAlphabet\mathbb{U}{msb}{m}{n}

\usepackage[arrow,matrix,curve]{xy}\SilentMatrices
\def\xyma{\xymatrix@M.7em}

\newtheorem{theorem}{Theorem}[section]

\newtheorem{lemma}[theorem]{Lemma}
\newtheorem{proposition}[theorem]{Proposition}
\theoremstyle{definition}

\usepackage{mathtools}

\def\CC{\mathcal C}
\def\FF{\mathcal F}
\def\GG{\mathcal G}
\def\lim{\sf Lim}

\def\Ab{\sf Ab}

\def\Gr{\sf Gr}

\usepackage[arrow,matrix,curve]{xy}\SilentMatrices
\def\xyma{\xymatrix@M.7em}

\begin{document}

\title{Dimension quotients as boundary limits}

\begin{abstract}
For a functor from the category of free presentations of a group to the category of all groups we define the boundary limit as an image of the natural map from limit to colimit. We show that the fourth dimension quotient of a group can be naturally described as the boundary limit of a simply-defined functor. 
\end{abstract}

\author{Roman Mikhailov}
\address{Saint Petersburg University, 7/9 Universitetskaya nab., St. Petersburg, 199034 Russia} 

\maketitle

\section{Introduction}
Theory of limits was developed in the series of papers \cite{IvanovMikhailov2015}, \cite{IM2017}, \cite{IMP}, \cite{MP2019}. It gives a method or a point of view on construction of functors in algebraic categories. In \cite{ICMTalk}, the limits are mentioned as a part of a program called {\it Speculative theory of functors}.

Recall the main concepts of the theory of limits. Let $G$ be a group. By $\textsf{Pres}(G)$ we denote the \textit{category of presentations} of $G$ with objects being free groups $F$ together with epimorphisms to $G$.
Morphisms are group homomorphisms over $G$. For a functor $\mathcal F:\textsf{Pres}(G)\to \textsf{Ab}$ from the category $\textsf{Pres}(G)$ to the category of abelian groups, one can
consider the (higher) limits ${\sf Lim}^i, \mathcal F,\ i\geq 0,$ over the category of presentations. That is, we fix our group $G$, consider free presentations $R\hookrightarrow F\twoheadrightarrow G$ and make functorial (on $F,R$) constructions $\mathcal F(F,R).$ The limits ${\sf Lim}^i\, \mathcal F(F,R),\ i\geq 0,$ will depend only on $G$. Moreover, the limits give a way to define functors from the category of all groups to abelian groups. 

The category $\textsf{Pres}(G)$ is strongly connected and has pair-wise coproducts. The limit ${\sf Lim} \mathcal F={\sf Lim}^0 \mathcal F$ has the following properties. For any $c\in \textsf{Pres}(G),$
\begin{equation}
{\sf Lim}\ \mathcal F=\{x\in \mathcal F(c)|\forall\,c'\in {\sf Pres}(G), \ \varphi, \psi:c\to c',\ \mathcal F(\varphi)(x)=\mathcal F(\psi)(x)\}.
\end{equation}
Moreover, it is known \cite{MP2019} that the limit of a functor from a strongly connected category with pair-wise coproducts is equal to the equalizer
\begin{equation}\label{equa}{\sf Lim}\ \mathcal F = \textrm{eq}(\mathcal F(c)\rightrightarrows \mathcal F(c\sqcup c))
\end{equation}
for any $c\in \textsf{Pres}(G)$.
 In particular, this equalizer does not depend on $c$. That is, a limit is not just an abstractly defined object, but is the equalizer of a pair of explicitly described maps. One more property of the limit, which makes it intuitively clear, is that ${\sf Lim}\ \mathcal F$ is the largest subgroup of $\mathcal F(c),$ which depends only on $G$. 

 The theory of (nonderived) limits can rewritten without any serious changes to the case of non-commutative values. For example, for functors $\mathcal F: {\sf Pres}(G)\to {\sf Gr}$, from the category of presentations of $G$ to the category of all groups, one can define the limit ${\sf Lim} \mathcal F$ via universal properties and prove the main statements like mentioned above. The limit ${\sf Lim}\ \mathcal F$ is the largest constant subgroup of $\mathcal F(c),$ for any $c\in {\sf Pres}(G).$ However, in the non-commutative case, there is a problem with definition of higher limits ${\sf Lim}^i \mathcal F,\ i\geq 1$. The first derived functor ${\sf Lim}^1 \mathcal F$ can be defined not a a group, but as a pointed set only.

 There is a large collection of examples of limits computed by various methods \cite{IvanovMikhailov2015}, \cite{IM2017}, \cite{IMP}, \cite{MP2019}. In particular, various combinations of derived functors and group homology appear as limits of simply defined functors in ${\sf Pres}(G)$. To illustrate briefly how it works, recall two examples from \cite{IvanovMikhailov2015} and \cite{MP2015}. For a group $H$, the lower central series are defined inductively as follows $$\gamma_1(H):=H,\ \gamma_{n+1}(H):=[\gamma_n(H),H]=\langle [x,y],\ x\in \gamma_n(H), y\in H\rangle^H,\ n\geq 1.$$ There is the following description of the (derived) limits: 
 \begin{align*}
& {\sf Lim}\  \gamma_3(F)/[\gamma_2(R),F]\gamma_4(F)=L_1{\sf S}^3(G_{ab}),\\ 
& {\sf Lim}\  \gamma_2(R)/[\gamma_2(R),F]=H_4(G;\mathbb Z/2)\ \ \ (\text{provided}\ G\ \text{is 2-torsion free}),\\ 
& {\sf Lim}^1\gamma_2(R)/[\gamma_2(R),F]=H_3(G;\mathbb Z/2)\ \ \ (\text{provided}\ G\ \text{is 2-torsion free}).
\end{align*}
Here $L_1{\sf S}^3$ is the first derived functor in the sense of Dold-Puppe of the symmetric cube, $G_{ab}:=G/\gamma_2(G)$ is the abelianization of $G$. 

It is easy to guess why the presented theory of limits belongs to the Speculative theory of functors. The symbols $F$ and $R$ play the role of building blocks or construction parts for design of tricky functors and natural transformations between them. Indeed, one can construct anything depending on $F$ and $R$, take (derived) limits and obtain the functor on $G.$

The theory of (derived) colimits in the categories of presentations is developed in \cite{IMS}. It is shown in that paper that a lot of interesting functors like cyclic homology or certain K-functors can be described via derived colimits. For a presentation $\mathcal F,$ the colimit ${\sf Colim}\ \mathcal F$ is the largest quotient of $\mathcal F(c),$ which depends only on $G$. 

In this paper, we will consider the ${\it boundary\  limits}$ and denote them as ${\sf Blim}.$ By definition, the boundary limit is the image of the limit in colimit, under the natural map, i.e. for a presentation $\mathcal F,$ 
$$
{\sf Blim}\ \mathcal F:={\sf Im}\{{\sf Lim}\ \mathcal F\to {\sf Colim}\ \mathcal F\}. 
$$
That is, for any $c\in {\sf Pres}(G),$ the boundary limit ${\sf Blim}\ \mathcal F$ is the image of the largest subgroup of $\mathcal F(c)$ which depends only on $G$ in the largest quotient of $\mathcal F(c)$ which depends only of $G$. Observe that, in the case of colimits, there is no analog of the formula (\ref{equa}), colimit clearly can be defined as a coequalizer of a set of maps, however, this set, in general, is infinite. 

The dimension subgroups are fundamental objects in the theory of group rings. In this paper, we will study them in the context of limits. Recall the definition. 
Let $\textbf{g}$ be the augmentation ideal of
$\mathbb Z[G]$. The subgroups $D_n(G):= G \cap (1 + \textbf{g}^n),\ n\geq 1$ are known as
\textit{dimension subgroups}. It is easy to see that, for any $G$ and $n\geq 1$, the
dimension subgroups contain the terms of lower central series of $G$:
$\gamma_n(G)\subseteq D_n(G)$.  The quotients $D_n(G)/\gamma_n(G)$ are called the {\it dimension quotients}. The structure of the dimension quotients in general is extremely complicated and, as it is shown recently in \cite{BartholdiMikhailov}, is related to the homotopy groups of spheres. 

In this paper we show (see Proposition \ref{pr11}) that, for every $n,$ there exists a natural monomorphism 
\begin{equation}\label{inclus}
{\sf Blim}\frac{F}{R'\gamma_n(F)}\hookrightarrow
D_n(G)/\gamma_n(G).
\end{equation}
Is it an isomorphism? For the moment we are not able to answer this question for general $n\geq 1.$ Since the dimension quotient may be nontrivial for $n\geq 4,$ the case $n=4$ turns out to be the first interesting step. We prove that, for $n=4$, the above map is a natural isomorphism, i.e. 
\begin{equation}\label{isom1}
{\sf Blim}\frac{F}{R'\gamma_4(F)}=D_4(G)/\gamma_4(G). 
\end{equation}
Observe that, this result follows from the results of the paper \cite{MP2019}, however the proofs from \cite{MP2019} are non-natural and based on a choose of basis in free groups and a routine analysis of forms of elements in free group rings, it is very complicated to generalize the arguments from \cite{MP2019} for higher $n$. In this paper we give a completely {\it functorial} proof of (\ref{isom1}). 

The fourth dimension quotient $D_4(G)/\gamma_4(G)$ is extremely complicated functor, as an abelian group it always has exponent 2, however, it's structure has relation to homotopy groups, in particular, the 2-torsion which appears in the fourth dimension quotient of some known examples is related to the homotopy group $\pi_4(S^2)=\mathbb Z/2$ (see \cite{BartholdiMikhailov}). Observe also that, the results of this paper can be easily reformulated for the case of Lie rings. In the Lie case, the dimension quotients are related to the derived functors of Lie functors (see \cite{BartholdiMikhailov}). 

It is natural to make a conjecture that the inclusion (\ref{inclus}) is an isomorphism for all $n$. What about boundary limits of the type ${\sf Blim} \frac{F}{\gamma_k(R)\gamma_n(F)},\ k<n?$ These functors may have a complicated nature, and be related to homotopy theory. 

The author thanks Sergei Ivanov for writting the Appendix containing an introduction to the theory of non-abelian ${\sf Lim}^1,$ which is needed for the results of this paper.

\section{Formulations and proofs}
For a free group $F,$ its normal subgroup $R$ and a quotient $G=F/R,$ denote by $\bf f$ the augmentation ideal of $\mathbb Z[F],$ and ${\bf r}:=(R-1)\mathbb Z[F].$ 

\begin{proposition}\label{pr11}
There exist the following natural exact sequence
\begin{multline*}
{\sf Blim}\ \frac{F}{R'\gamma_n(F)}\hookrightarrow{}D_n(G)/\gamma_n(G)\to {\sf Lim}^1\frac{R\cap (1+{\bf rf}+{\bf f}^n)}{R'(R\cap \gamma_n(F))}\twoheadrightarrow\\ 
{\sf Lim}^1\frac{F\cap (1+{\bf rf}+{\bf f}^n)}{R'\gamma_n(F)}
\end{multline*}
where the last surjective map is a map between pointed sets. 
\end{proposition}

\begin{proof}
Observe that (see, for example, page 80 in \cite{Gupta:1987}), for all $n\geq 1$, there exists the following natural short exact sequence
\begin{equation}\label{eqr1}
\frac{R\cap (1+{\bf rf}+{\bf f}^n)}{\gamma_2(R)(R\cap\gamma_n(F))}\hookrightarrow  \frac{F\cap (1+{\bf rf}+{\bf f}^n)}{\gamma_2(R)\gamma_n(F)}\twoheadrightarrow \frac{D_n(G)}{\gamma_n(G)}
\end{equation}
Using the description of the dimension quotient $\frac{D_n(G)}{\gamma_n(G)}=\frac{F\cap (1+{\bf r}+{\bf f}^n)}{R\gamma_n(F)},$
we see that the sequence (\ref{eqr1}) is central. Therefore, by Lemma \ref{central}, there is the following natural exact sequence 
\begin{multline}\label{mult1}
{\sf Lim}\frac{R\cap (1+{\bf rf}+{\bf f}^n)}{\gamma_2(R)(R\cap\gamma_n(F))}\hookrightarrow  {\sf Lim}\frac{F\cap (1+{\bf rf}+{\bf f}^n)}{\gamma_2(R)\gamma_n(F)}\to  \frac{D_n(G)}{\gamma_n(G)}\to\\ 
{\sf Lim}^1\frac{R\cap (1+{\bf rf}+{\bf f}^n)}{\gamma_2(R)(R\cap\gamma_n(F))}\twoheadrightarrow  {\sf Lim}^1\frac{F\cap (1+{\bf rf}+{\bf f}^n)}{\gamma_2(R)\gamma_n(F)}
\end{multline}
where the last arrow is a surjective map of pointed sets. We use an elementary property that ${\sf Lim}^1$ of a constant functor (i.e. which depends only on $G$) is trivial. 

Now observe that, there is a natural inclusion
$$
\frac{F}{F\cap (1+{\bf rf}+{\bf f}^n)}\subset \frac{\bf f}{{\bf rf}+{\bf f}^n}$$ induced by the map $f\mapsto f-1\in {\bf f},\ f\in F$. The functor $\frac{\bf f}{{\bf rf}+{\bf f}^n}=G_{ab}\otimes {\bf f}/{\bf f}^{n-1}$ is monoadditive and 
$$
{\sf Lim} \frac{\bf f}{{\bf rf}+{\bf f}^n}=0. 
$$
(Moverover, all higher limits of this presentation are zero as well, see \cite{IvanovMikhailov2015} for discussion of such examples.) Since the limit is left exact, we conclude that, there are natural isomorphisms 
\begin{align*}
& {\sf Lim} \frac{R\cap (1+{\bf rf}+{\bf f}^n)}{\gamma_2(R)(R\cap\gamma_n(F))}={\sf Lim}\frac{R}{R'\gamma_n(F)}\\ 
& {\sf Lim}\frac{F\cap (1+{\bf rf}+{\bf f}^n)}{\gamma_2(R)\gamma_n(F)}={\sf Lim}\frac{F}{R'\gamma_n(F)}. 
\end{align*}
Now observe that $${\sf Colim}\frac{F}{R'\gamma_n(F)}=G/\gamma_n(G)$$
and that 
$$
\frac{R}{R'(R\cap\gamma_n(F))}\cap {\sf Lim}\frac{F}{R'\gamma_n(F)}={\sf Lim}\frac{R}{R'(R\cap \gamma_n(F))}
$$
where all limits are considered as subgroups of $F/R'\gamma_n(F).$ Therefore, there is the natural short exact sequence 
$$
{\sf Lim}\frac{R}{R'(R\cap\gamma_n(F))}\hookrightarrow {\sf Lim}\frac{F}{R'\gamma_n(F)}\twoheadrightarrow {\sf Blim} \frac{F}{R'\gamma_n(F)}
$$
and the needed statement follows from the exact sequence
(\ref{mult1}). \end{proof}

\begin{proposition}\label{inj}
There is the following natural short exact sequence 
$$
\frac{\gamma_3(F)}{[R,F']\gamma_4(F)}\hookrightarrow {\sf S}^2(G_{ab})\otimes F_{ab}\twoheadrightarrow S^3(G_{ab}). 
$$
Here ${\sf S}^*$ are symmetric powers. 
\end{proposition}
\begin{proof}
We have the following diagram
$$
\xyma{\frac{{\sf S}^2(G_{ab})\otimes R/(R\cap F')}{{\sf Tor}({\sf S}^2(G_{ab}), G_{ab})}\ar@{>->}[d]\ar@{=}[r] & \frac{{\sf S}^2(G_{ab})\otimes R/(R\cap F')}{{\sf Tor}({\sf S}^2(G_{ab}), G_{ab})}\ar@{>->}[d]\\ 
Ker\ar@{>->}[r] \ar@{->>}[d]& {\sf S}^2(G_{ab})\otimes F_{ab}\ar@{->>}[r] \ar@{->>}[d] & {\sf S}^3(G_{ab})\ar@{=}[d]\\ 
{\sf L}^3(G_{ab})\ar@{>->}[r] & {\sf S}^2(G_{ab})\otimes G_{ab}\ar@{->>}[r] & {\sf S}^3(G_{ab}).
}
$$
where ${\sf L}^3$ is the third Lie cube. The map 
$$
\frac{\gamma_3(F)}{[R,F']\gamma_4(F)}\to Ker
$$
is given as 
$$
[[a,b],c]\mapsto \bar a\bar c\otimes b-\bar b\bar c\otimes a,\ a,b,c\in F. 
$$
It is east to check that the map is well-defined. Recall that there is a natural descrition of the Lie functor
$$
{\sf L}^3(G_{ab})=\frac{\gamma_3(F)}{[[R,F],F]\gamma_4(F)}. 
$$
We obtain a natural short exact sequence 
\begin{equation}\label{ker}
\frac{[[R,F],F]\gamma_4(F)}{[R,F']\gamma_4(F)}\hookrightarrow \frac{\gamma_3(F)}{[R,F']\gamma_4(F)}\twoheadrightarrow {\sf L}^3(G_{ab}). 
\end{equation}
The map 
$$
\frac{[[R,F],F]\gamma_4(F)}{[R,F']\gamma_4(F)}\to \frac{{\sf S}^2(G_{ab})\otimes R/(R\cap F')}{{\sf Tor}({\sf S}^2(G_{ab}), G_{ab})}
$$
is an epimorphism, since every term $ab\otimes r,\ a,b\in G_{ab},\ r\in R$ is an image of an element $[[r,a'],b'],$ where $a', b'$ are preimages of $a,b$ in $F.$ Now observe that there exists a natural inverse surjective map 
$$
\frac{{\sf S}^2(G_{ab})\otimes R/(R\cap F')}{{\sf Tor}({\sf S}^2(G_{ab}), G_{ab})}\to 
\frac{[[R,F],F]\gamma_4(F)}{[R,F']\gamma_4(F)}, 
$$
defined as 
$$
ab\otimes r\mapsto [[r,a'],b'],
$$
where $a,b,a',b',r$ are as above. It is easy to check that it is well-defined, therefore, the map (\ref{ker}) is an isomorphsim. 
\end{proof}

Using elementary properties of higher limits (see \cite{IvanovMikhailov2015}), we conclude 
\begin{align*}
& {\sf Lim}^i\ \frac{\gamma_3(F)}{[R,F']\gamma_4(F)}=0,\ i\neq 1\\ 
& {\sf Lim}^1\ \frac{\gamma_3(F)}{[R,F']\gamma_4(F)}={\sf S}^3(G_{ab}),\ i\neq 1. 
\end{align*}
Observe also that there exist the following exact sequence
$$
{\sf Tor}({\sf S^2(G_{ab})},G_{ab})\hookrightarrow {\sf S}^2(G_{ab})\otimes R/(R\cap F')\to \frac{\gamma_3(F)}{[R,F']\gamma_4(F)}\twoheadrightarrow {\sf L}^3(G_{ab}). 
$$

Now we formulate our main result
\begin{theorem}
There is a natural isormorphism
$$
{\sf Blim}\frac{F}{R'\gamma_4(F)}=D_4(G)/\gamma_4(G). 
$$
\end{theorem}
\begin{proof}
We will prove that 
$$
{\sf Lim}^1\frac{R\cap (1+{\bf rf}+{\bf f}^4)}{R'(R\cap \gamma_4(F))}=0
$$
and the needed result will follow from Proposition \ref{pr11}. We will show that the quotient
$$
\frac{F\cap (1+{\bf rf}+{\bf f}^4)}{R'\gamma_4(F)}
$$
is constant, therefore, the subquotient generated by $R\cap (1+{\bf rf}+{\bf f}^4)$ is constant as well and the result follows. (A subpresentation of a constant or an epimorphic image of a constant functor is constant, see \cite{MP2019}.)  

Recall that, there is a natural exact sequence 
$$
L_1{\sf S}^2(G_{ab})\hookrightarrow \frac{F'}{R'\gamma_3(F)}\to \frac{{\bf f}^2}{{\bf rf}+{\bf f}^3}
$$
Here $L_1{\sf S}^2$ is the first derived functor in the sense of Dold-Puppe of the symmetric square. This sequence appeared in \cite{HMP} and \cite{MPEMS} in a more generalized context, its proof is completely functorial. Since $L_1{\sf S}^2(G_{ab})$ is constant, the natural commutative diagram

$$
\xyma{\frac{F'}{R'\gamma_4(F)} \ar@{->}[r] \ar@{->>}[d] & \frac{{\bf f}^2}{{\bf rf}+{\bf f}^4}\ar@{->>}[d]\\ \frac{F'}{R'\gamma_3(F)} \ar@{->}[r] & \frac{{\bf f}^2}{{\bf rf}+{\bf f}^3}}
$$
shows that it is enough to consider the kernel of the natural map 
$$
\frac{R'\gamma_3(F)}{R'\gamma_4(F)}\to \frac{{\bf rf}+{\bf f}^3}{{\bf rf}+{\bf f}^4}
$$
Our proof will be complete as soon as we will prove that the kernel of this map is constant. 

Obseve that, 
$$
R'\cap \gamma_3(F)=[R\cap F', R].
$$
This is a well-known description and can be proved in different ways. For example, one can prove it using the properties of exterior squares $\Lambda^2(F)$ and $\Lambda^2(R/R\cap F')$ (see \cite{MP2019} for details). 

Observe that ${\bf rf}\cap {\bf f}^3=({\bf r}\cap {\bf f}^2){\bf f}.$ Therefore, there is a natural isomorphism 
$$
\frac{{\bf f}^3}{{\bf fr}\cap {\bf f}^3+{\bf f}^4}={\bf g}^2/{\bf g}^3\otimes F_{ab}. 
$$
Consider the following commutative square with natural maps
$$
\xyma{\frac{R'\gamma_3(F)}{R'\gamma_4(F)}\ar@{->}[r] \ar@{->>}[d] & {\bf g}^2/{\bf g}^3\otimes F_{ab}\ar@{->>}[d]\\ \frac{\gamma_3(F)}{[R,F']\gamma_4(F)}\ar@{->}[r] & {\sf S}^2(G_{ab})\otimes F_{ab}}
$$
Since the lower horizontal map is injective by Proposition \ref{inj}, we conclude that, the kernel of the upper horizontal map is contained in the subgroup $$\frac{[R,F']\gamma_4}{R'\cap \gamma_3(F)\gamma_4(F)}\subset \frac{R'\gamma_3(F)}{R'\gamma_4(F)}.$$
Now consider the natural commutative square with obviously defined maps
$$
\xyma{\gamma_2(G)/\gamma_3(G)\otimes R/(R\cap F')\ar@{->}[r] \ar@{->>}[d] & {\bf g}^2/{\bf g}^3\otimes F_{ab}\ar@{=}[d]\\ \frac{[R,F']\gamma_4(F)}{(R'\cap \gamma_3(F))\gamma_4(F)}\ar@{->}[r] & {\bf g}^2/{\bf g}^3\otimes F_{ab}}
$$
The kernel of the lower horizontal map is an epimorphic image of the kernel of the upper horizontal map. Since the groups $F_{ab}$ and $R/(R\cap F')$ are free abelain, the kernel of the upper map is a subgroup of 
$
{\sf Tor}({\bf g}^2/{\bf g}^3, G_{ab}),
$ i.e. it is a constant functor and the proof is complete.  
\end{proof}

\section{Appenix. Non-abelian ${\sf Lim}^1$} 

This Appendix is written by Sergei Ivanov. 

In this section, $\CC$ denotes a small category.  The main aim of this section is to give the definition of a pointed set ${\sf Lim}^1\mathcal F$ for a functor to the category 
of groups $\mathcal F:{\sf C}\to {\sf Gr}$ and to describe some of its properties (Propositions \ref{prop:exact_seq_1}, \ref{prop:exact_seq_2}). The pointed set $\lim^1$ is a generalization of the first  cohomology set  $H^1(G,A)$ of a group $G$ with non-abelian group of coefficients $A.$ The content of this section is a strightforward generalization of the theory of group cohomology with non-abelian coefficients that can be found in \cite[Ch.5]{Serre}. See also \cite[Ch.IX \S 4]{BK}

\subsection{Derived limits of abelian functors}

Then for a functor to the category of abelian groups $\FF:\CC \to \Ab$ the limit $\lim \FF$ exists and the functor 
\begin{equation}
    \lim : \Ab^\CC \longrightarrow \Ab
\end{equation}
is an additive left exact functor between abelian categories. Since $\CC$ is small, $\Ab^\CC$ has enough of projectives and injectives.  Then the limit functor has right derived functors
\begin{equation}
\lim^n := {\bf R}^n \lim : \Ab^\CC \longrightarrow \Ab.
\end{equation}
In particular, for any short exact sequence of functors $0\to \FF' \to \FF \to \FF'' \to 0$ there is a long exact sequence
\begin{equation}
0 \to \lim\: \FF' \to \lim\: \FF \to \lim\: \FF'' \to \lim^1 \FF' \to \lim^1 \FF \to \lim^1 \FF'' \to \dots. 
\end{equation}

\subsection{Standard complex}

Denote by $N\CC$ the nerve of the category $\CC.$ The elements of $(N\CC)_n$ are  sequences of morphisms 
 $(\alpha_1,\dots,\alpha_n)$ such that ${\sf dom}( \alpha_i )={\sf cod}(\alpha_{i+1})$
\begin{equation}
\bullet  \overset{\alpha_1}\longleftarrow \bullet   \overset{\alpha_2}\longleftarrow \dots  \overset{\alpha_n}\longleftarrow \bullet
\end{equation}
and $(N\CC)_0={\sf Ob}(\CC).$ The face maps are defined as follows
\begin{align}
d_0(\alpha_1,\dots,\alpha_n) &= (\alpha_2,\dots, \alpha_n) \\
d_{i}(\alpha_1,\dots,\alpha_n) &= (\alpha_1,\dots, \alpha_i\alpha_{i+1}, \dots,  \alpha_n), \hspace{1cm} 1\leq i \leq n-1 \\
d_n(\alpha_1,\dots,\alpha_n) &= (\alpha_1,\dots, \alpha_{n-1}),
\end{align}
for $n\geq 2$ and $d_0(\alpha_1)={\sf cod}(\alpha_1),$ $d_1(\alpha_1)={\sf dom}(\alpha_1).$
Degeneracy maps are defined by 
$
s_i(\alpha_1,\dots,\alpha_n ) = (\alpha_1,\dots,\alpha_i, {\sf id}, \alpha_{i+1},\dots,\alpha_n).
$

Let  $\FF:\CC\to \Ab$ be a functor. For an element  $x\in \FF(c)$ and a morphism $\alpha:c\to c'$ we set 
\begin{equation}
\alpha \cdot x := \FF(\alpha)(x) \in \FF(c').     
\end{equation}
For $(\alpha_1,\dots,\alpha_n)\in (N\CC)_n$ we set 
\begin{equation}
    \FF(\alpha_1,\dots,\alpha_n) := \FF( {\sf cod}(\alpha_1)),
\end{equation}
where ${\sf cod}(\alpha)$ denotes the codomain of $\alpha.$
Consider the cochain complex $C^*(\CC,\FF)$ 
whose components are given by
\begin{equation}
C^n(\CC,\FF)= \prod_{\alpha \in (N\CC)_n} \FF(\alpha). 
\end{equation} For an element $a\in C^n(\CC,\FF)$ we denote by $a(\alpha) \in \FF(\alpha)$ the corresponding component.  Then  the differential is given by $\partial^n =\sum_{i=0}^{n+1} (-1)^i d^i,$ where
\begin{align}
d^0(a)(\alpha) &= \alpha_1\cdot  a(d_0(\alpha)),
\\
d^i(a)(\alpha) &= a(d_i(\alpha)), \hspace{1cm} 1 \leq i \leq n+1.
\end{align}

\begin{proposition}[{\cite[A.II,Prop.3.3]{GZ}, \cite[Prop.2.6]{Gr}}]
For any small category $\CC$ and any functor $\FF:\CC\to \Ab$ there is a natural isomorphism
\begin{equation}
\lim^n \FF \cong  H^n( C^*(\CC,\FF)).    
\end{equation}
\end{proposition}

\subsection{Abelian $1$-cocycles and $1$-coboundaries} We denote by $Z^n(\CC,\FF)$ the set of $n$-cocycles of the complex $C^*(\CC,\FF)$ and by $B^n(\CC,\FF)$ the set of $n$-coboundaries. An element of the group $Z^n(\CC,\FF)$ is called  $n$-cocycle of $\CC$ with coefficients in $\FF$ and an element of $B^n(\CC,\FF)$ is called $n$-coboundary of $\CC$ with coefficients in $\FF.$ 

\begin{proposition}\label{prop:1-coc_ab}
Let $a\in \prod_{\alpha\in {\sf Mor}(\CC)} \FF({\sf cod}(\alpha)).$
\begin{enumerate}
\item $a\in Z^1(\CC,\FF) $ if and only if
\begin{equation*}\label{eq:1-coc_ab}
a(\alpha \beta) = a(\alpha) + \alpha \cdot a(\beta).  
\end{equation*} 
\item  $a\in B^1(\CC,\FF)$  if and only if there exists $x\in \prod_{c\in {\sf }\CC}\FF(c)$ such that
\begin{equation*}
a(\alpha) = x( {\sf cod}(\alpha)) - \alpha \cdot x({\sf dom}(\alpha)).
\end{equation*}

\item There is an isomorphism 
\begin{equation*}
\lim^1\FF \cong \frac{Z^1(\CC,\FF)}{B^1(\CC,\FF)}
\end{equation*}

\item two $1$-cocycles $a,a'\in Z^1(\CC,\FF)$ represent the same element in $\lim^1 \FF$ if and only if there exists $x\in \prod_{c\in {\sf Ob}(\CC)} \FF(c)$ such that 
\begin{equation*}\label{eq:equivalence_ab}
a'(\alpha) = -x({\sf cod}(\alpha)) + a(\alpha) + \alpha\cdot x({\sf dom}(\alpha)).  
\end{equation*}
\end{enumerate}
\end{proposition}
\begin{proof}
(1) and (2) can be proved by a direct computation using the definition of $C^*(\CC,\FF)$. (3) and (4) are obvious. 
\end{proof}

\subsection{Non-abelian $\lim^1$}
Denote by $\Gr$ the category of groups and consider a functor $\FF:\CC \to \Gr.$ For an element $x\in\FF(c)$ and a morphism $\alpha:c\to c'$ we set 
\begin{equation}
    {}^\alpha x = \FF(\alpha)(x).
\end{equation}
It is well known that $\lim\;\FF$ can be described as follows 
\begin{equation}
\lim\;\FF = \{ x\in \prod_c \FF(c) \mid x({\sf cod}(\alpha)) = {}^\alpha x({\sf dom}(\alpha)) \}.    
\end{equation}
Let us define $\lim^1\FF$ in a similar manner.

Denote by $Z^1(\CC,\FF)$ the set of elements $a\in \prod_{\alpha} \FF({\sf cod}(\alpha))$ such that 
\begin{equation}
    a(\alpha \beta)=a(\alpha) \cdot {}^\alpha a(\beta).
\end{equation}
We consider the set $Z^1(\CC,\FF)$ as a pointed set with the basepoint $a(\alpha)=1.$
The group $\prod_{c} \FF(c)$ acts on the set $Z^1(\CC,\FF)$ as follows 
\begin{equation}
(a^x)(\alpha) = x({\sf cod}(\alpha))^{-1} \cdot  a(\alpha) \cdot {}^\alpha x({\sf dom}(\alpha)) 
\end{equation}
Then we define $\lim^1 \FF$ as the quotient by this action 
\begin{equation}
\lim^1\: \FF := Z^1(\CC,\FF) / \prod_{c} \FF(c).     
\end{equation}
Proposition \ref{prop:1-coc_ab} implies that this definition is a generalization of the definition for the functors to the category of abelian groups. Moreover, it is easy to see that $\lim^1\FF$ is natural in $\FF$ i.e. a natural transformation $\FF'\to \FF$ defines a map $\lim^1\FF' \to \lim^1\FF.$ 

Let $\FF:\CC\to {\sf Gr}$ be a functor and  $\GG\subseteq \FF$ be a subfunctor.  Then the quotient $\FF/\GG$ defines a functor to the category of poited sets $\FF/\GG: \CC\to {\sf Sets}_*,$ where $(\FF/\GG)(c)=\FF(c)/\GG(c)$ consists of left cosets $x\cdot \GG(c), x\in \FF(c).$ 

\begin{lemma}\label{central}
There is a well defined map of pointed sets
\[
\delta : \lim\: \FF/\GG \longrightarrow \lim^1 \GG 
\]
such that for any $x\in \lim \: \FF/\GG\subseteq \prod (\FF/\GG)(c)$ and any its preimage $\bar x\in \prod \FF(c)$ the class $\delta(x)$ is defined by $1$-cocycle $d(\bar x)\in Z^1(\CC,\GG)$ given by 
\[
d( \bar x)(\alpha) = \bar x^{-1}({\sf cod}(\alpha)) \cdot {}^\alpha \bar x({\sf dom}(\alpha)).
\]
Moreover, if $\GG$ is a central subgroup  of $\FF,$ then $\delta$ is a homomorphism.
\end{lemma}
\begin{proof}
A straightforward computation shows that $d(\bar x)$ is $1$-cocycle. If $\bar x, \tilde x$ are two preimages of $x,$ then $\bar x = \tilde x y$ for some $y\in \prod_c \GG(c)$ and this implies that $d(\bar x)$ is equivalent to $d(\tilde x).$

Now assume that $\GG$ is a central subgroup and $x_1,x_2\in \lim \FF/\GG.$ Then 
\begin{equation}
\begin{split}
d(\bar x_1 \bar x_2)(\alpha) & = \bar x_2^{-1}({\sf cod}(\alpha))  \cdot \Big( \bar x_1^{-1} ({\sf cod}(\alpha)) \cdot {}^\alpha \bar x_1({\sf dom}(\alpha)) \Big) \cdot {}^{\alpha} \bar x_2( {\sf dom}(\alpha) ) \\ 
&=\Big( \bar x_1^{-1} ({\sf cod}(\alpha)) \cdot {}^\alpha \bar x_1({\sf dom}(\alpha)) \Big) \cdot \Big( \bar x_2^{-1} ({\sf cod}(\alpha)) \cdot {}^\alpha \bar x_2({\sf dom}(\alpha)) \Big)\\
&= d(\bar x_1)(\alpha) \cdot d(\bar x_2)(\alpha).
\end{split}
\end{equation}
It follows that $\delta$ is a homomorphism. 
\end{proof}

\begin{proposition}[cf. {\cite[Prop.36]{Serre}}]\label{prop:exact_seq_1}
Let $\FF:\CC\to \Gr $ be a functor and let $\GG\subseteq \FF$ be a subfunctor. Then the sequence of pointed sets
\begin{equation*}
    1 \to \lim\: \GG \to \lim\: \FF \to \lim\: \FF/\GG \to \lim^1 \GG \to \lim^1 \FF
\end{equation*}
is exact.
\end{proposition}
\begin{proof}
Since the functor $\lim$ is right adjoint, it commutes with equializers, and hence the sequence $1 \to \lim\: \GG \to \lim\: \FF \to \lim\: \FF/\GG$ is exact. The equation ${\sf Ker}(\lim^1\GG\to \lim^1\FF)={\sf Im}(\delta)$ follows obviously from the definition of $\delta$.

Prove that ${\sf Ker}(\delta)={\sf Im}(\lim\:\FF\to \lim\: \FF/\GG).$ Let $x\in \lim\:\FF/\GG$ and $\bar x\in \prod \FF(c)$ is its preimage. Then $\delta(x)$ is trivial if and only if there exists $y\in \prod \GG(c)$ such that 
\[\bar x^{-1}({\sf cod}(\alpha) )\cdot {}^\alpha \bar x({\sf dom}(\alpha))  =  y^{-1}({\sf cod}(\alpha) )\cdot {}^\alpha y({\sf dom}(\alpha)).\] 
The equation is equivalent to the following equation 
\[
(\bar x y^{-1})({\sf cod}(\alpha)) = {}^\alpha (\bar x y^{-1})({\sf dom}(\alpha))
\]
which is equivalent to the fact that $\bar x y^{-1}\in \lim \:\FF.$ It follows that $\delta(x)$ is trivial if an only if $x$ has a preimage in $\lim\:\FF.$
\end{proof}

\begin{proposition}[{cf. \cite[Prop.38]{Serre}, \cite[Prop.IX.2.3]{BK}}]\label{prop:exact_seq_2}
For any short exact sequence
\[
1 \longrightarrow\FF' \longrightarrow \FF \longrightarrow \FF'' \longrightarrow 1
\]
of functors $\CC\to \Gr$ the sequence of pointed sets
\[
1 \to \lim\: \FF' \to \lim\: \FF \to \lim\: \FF'' \to \lim^1\: \FF' \to \lim^1\: \FF \to \lim^1\: \FF'' 
\]
is exact. 
\end{proposition}
\begin{proof}
By Proposition \ref{prop:exact_seq_1} we have the exactness in all terms except the last one.  Prove that ${\sf Ker}(\lim^1\FF\to \lim^1\FF'')={\sf Im}(\lim^1\FF' \to \lim^1 \FF).$

Denote the epimorphism by  $f:\FF\to \FF''$ and assume that $\FF'\subseteq \FF.$ Take a $1$-cocycle $a\in Z^1(\CC,\FF).$ Then it corresponds to an element of the kernel if and only if 
\[
f(a)(\alpha)=x({\sf cod}(\alpha))^{-1} \cdot {}^\alpha x({\sf dom}(\alpha)) \] 
for some $x\in\prod_c  \FF''(c).$ Equivalently this can be rewritten as
\[x({\sf cod}(\alpha)) \cdot g(a)(\alpha)= {}^\alpha x({\sf dom}(\alpha)).\]
If we take a preimage $\bar x$ of $x$ we see that it is equivalent to the fact that 
\[\bar x({\sf cod}(\alpha)) \cdot a(\alpha)= b(\alpha)\cdot  {}^\alpha \bar x({\sf dom}(\alpha))\]
for some $\bar x\in \prod_c \FF(c)$ and $b\in \prod_\alpha \FF'({\sf cod}(\alpha)).$ Hence the fact that the class $a$ is in ${\sf Ker}(\lim^1 \FF \to \lim^1 \FF'' )$ is equivalent to the fact that there exists $\bar x\in \prod_c \FF(c)$ and $b\in \prod_\alpha \FF'({\sf cod}(\alpha))$ such that 
\[a(\alpha) = \bar x({\sf cod}(\alpha))^{-1} \cdot  b(\alpha)\cdot {}^\alpha \bar x({\sf dom}(\alpha)).\]
And it is easy to see that this is equiavalent to the fact that the class of $a$ is in  ${\sf Im}(\lim^1\FF' \to \lim^1\FF ).$
\end{proof}

\noindent
{\bf Acknowledgement.} The work was performed at the Saint Petersburg Leonhard Euler International Mathematical Institute and supported by the Ministry of Science and Higher Education of the Russian Federation (agreement no. 075–15–2022–287).


\begin{thebibliography}{}
\bibitem{BartholdiMikhailov} L. Bartholdi and R. Mikhailov: Group and Lie algebra filtrations and homotopy groups of spheres, J. Topology, \textbf{16} (2023), 822--853. 


\bibitem{BK} A.K. Bousfield and D.M. Kan: Homotopy limits, completions and localizations. Vol. 304. Springer Science and Business Media, 1972.

\bibitem{GZ} P, Gabriel and M. Zisman: Calculus of fractions and homotopy theory. Vol. 35. Springer Science and Business Media, 2012.


\bibitem{Gr} J. Grodal: Higher limits via subgroup complexes, {\it Ann. Math.} {\bf 155} (2002), 405-–457.

\bibitem{Gupta:1987} Narain Gupta: \textit{Free Group Rings}, Contemporary Mathematics, Vol. \textbf{66}, American Mathematical Society, 1987.


\bibitem{HMP} M. Hartl, R. Mikhailov and I.B.S. Passi: Dimension quotients, {\it Journal of Indian Math. Soc.} (2009), 63--107. 

\bibitem{IvanovMikhailov2015} S.O. Ivanov and R. Mikhailov: A higher limit approach to homology theories, \textit{J. Pure Appl.
Algebra} \textbf{219} (2015), 1915--1939.


\bibitem{IM2017} S.O. Ivanov and R. Mikhailov: Higher limits, homology theories and frcodes. In: Combinatorial And Toric Homotopy: Introductory Lectures 35 (2017). 

\bibitem{IMP} S.O. Ivanov, R. Mikhailov and F. Pavutnitsky: Limits, standard complexes and fr-codes,
\textit{Mat.Sb.} \textbf{211} (2020), 1568--1591.

\bibitem{IMS} S.O. Ivanov, R. Mikhailov and V. Sosnilo: Higher colimits, derived functors and homology,
\textit{Mat. Sb.} \textbf{210} (2019), 1222--1258.

\bibitem{MP2015} R. Mikhailov and I.B.S. Passi: Generalized dimension subgroups and derived functors, \textit{J. Pure Appl. Algebra} \textbf{220} (2016), 2143–-2163.

\bibitem{MPEMS} R. Mikhailov and I.B.S. Passi: Free group rings and derived functors, Proceedings ECM-2018, ECM  407-425 (2018).

\bibitem{MP2019} R. Mikhailov and I.B.S. Passi: Dimension quotients, Fox subgroups and limits of functors, \textit{Forum Math.} \textbf{31} (2019), 385--401.

\bibitem{ICMTalk} R. Mikhailov: Homotopy patterns in group theory, {Proc. ICM2022}, arxiv:2111.00737




\bibitem{Serre} J.-P. Serre: Galois cohomology. Springer Science and Business Media, 2013.

\end{thebibliography}
\end{document}